\documentclass[a4paper,12pt,reqno]{article}
\usepackage[latin1]{inputenc}
\usepackage[T1]{fontenc}
\usepackage{lmodern}
\usepackage[english]{babel}
\usepackage{amsmath}
\usepackage{amssymb}
\usepackage{amsthm}
\usepackage{bm}
\usepackage[pagestyles]{titlesec}\usepackage[pdfpagemode=UseNone,pdfstartview=FitH]{hyperref}
\renewpagestyle{plain}{\setfoot[\thepage][][]{}{}{\thepage}}
\theoremstyle{plain}
\newtheorem{theorem}{Theorem}[section]
\newtheorem{corollary}{Corollary}[section]

\theoremstyle{definition}
\newtheorem{definition}{Definition}[section]
\newtheorem{example}{Example}[section]
\newtheorem{remark}{Remark}[section]
\numberwithin{equation}{section}
\allowdisplaybreaks[1]
  
\newcommand*{\Zset}{\mathbb{Z}}

\newcommand*{\Cset}{\mathbb{C}}  
\begin{document}
\title{\textbf{Determinant and inverse of join matrices on two sets}}
\author{}
\date{}
\maketitle
\begin{center}
\textsc{Mika Mattila and Pentti Haukkanen$^*$}\\
School of Information Sciences\\
FI-33014 University of Tampere, Finland\\[5mm]
May 31, 2011\\[1cm]
In honor of Abraham Berman, Moshe Goldberg, and Raphael Loewy\\[1cm]
\end{center}
{\bf Abstract} Let $(P,\preceq)$ be a lattice and $f$ a complex-valued function on $P$. We define meet and join matrices on two arbitrary subsets $X$ and $Y$ of $P$ by $(X,Y)_f=(f(x_i\wedge y_j))$ and $[X,Y]_f=(f(x_i\vee x_j))$ respectively. Here we present expressions for the determinant and the inverse of $[X,Y]_f$. Our main goal is to cover the case when $f$ is not semimultiplicative since the formulas presented earlier for $[X,Y]_f$ cannot be applied in this situation. In cases when $f$ is semimultiplicative we obtain several new and known formulas for the determinant and inverse of $(X,Y)_f$ and the usual meet and join matrices $(S)_f$ and $[S]_f$. We also apply these formulas to LCM, MAX, GCD and MIN matrices, which are special cases of join and meet matrices.
\\[5mm]
\emph{$^\ast$Corresponding\ author.}\ \textit{Tel.:}\ +358\ 31\ 3551\ 7030,\ \textit{fax:}\ +358\ 31\ 3551\ 6157\\
\emph{E-mail addresses:}\ mika.mattila@uta.fi\ (M. Mattila),\ pentti.haukkanen@uta.fi\ (P. Haukkanen)

\newpage

\section{Introduction}

Let $S=\{x_1,x_2,\ldots,x_n\}$ be a set of distinct positive integers, and let $f$ be an arithmetical function. Let $(S)_f$ denote the $n\times n$ matrix having $f((x_i,x_j))$, the image of the greatest common divisor of $x_i$ and $x_j$, as its $ij$ entry. Analogously, let $[S]_f$ denote the $n\times n$ matrix having $f([x_i,x_j])$, the image of the least common multiple of $x_i$ and $x_j$, as its $ij$ entry. That is, $(S)_f=(f((x_i,x_j)))$ and $[S]_f=(f([x_i,x_j]))$. The matrices $(S)_f$ and $[S]_f$ are referred to as the GCD and LCM matrices on $S$ associated with $f$, respectively. The set $S$ is said to be GCD-closed if $(x_i,x_j)\in S$ whenever $x_i,x_j\in S$, and the set $S$ is said to be factor-closed if it contains every divisor of $x$ for any $x\in S$. Clearly every factor-closed set is GCD-closed but the converse does not hold. 

In 1875 Smith \cite{Sm} calculated $\det(S)_f$ when $S$ is factor-closed and $\det[S]_f$ in a more special case. Since then a large number of results on GCD and LCM matrices have been presented in the literature. See, for example \cite{A, BeL, BD, BL, HS, HWS, Ho, HE, HL1, HL2, L}. Haukkanen \cite{H} generalized the concept of a GCD matrix into a meet matrix and later Korkee and Haukkanen \cite{KH2} did the same with the concepts of LCM and join matrices. These generalizations happen as follows.

Let $(P,\preceq)$ be a locally finite lattice, let $S=\{x_1,x_2,\ldots,x_n\}$ be a subset of $P$ and let $f$ be a complex-valued function on $P$. The $n\times n$ matrix $(S)_f=(f(x_i\wedge x_j))$ is called the meet matrix on $S$ associated with $f$ and the $n\times n$ matrix $[S]_f=(f(x_i\vee x_j))$ is called the join matrix on $S$ associated with $f$. If $(P,\preceq)=(\Zset^+,|)$, then meet and join matrices become respectively ordinary GCD and LCM matrices. However, some additional assumptions regarding the lattice $(P,\preceq)$ are still needed and we analyse these in Section 2. 

The properties of meet and join matrices have been studied by many authors (see, e.g., \cite{AST, H, HWS, HoS, IHM, KH2, KH3, KH5, Lu, O, R, SC}). Haukkanen \cite{H} calculated the determinant of $(S)_{f}$ on an arbitrary set $S$ and obtained the inverse of $(S)_{f}$ on a lower-closed set $S$ and Korkee and Haukkanen \cite{KH1} obtained the inverse of $(S)_{f}$ on a meet-closed set $S$. Korkee and Haukkanen \cite{KH2} present, among others,
formulas for the determinant and inverse of $[S]_{f}$ on meet-closed, join-closed, lower-closed
and upper-closed sets $S$.

Most recently,  Altinisik, Tuglu and Haukkanen \cite{ATH} generalized the concepts of meet and join matrices and defined meet and join matrices on two sets. Later these matrices were also treated in \cite{KH4}. Next we present the same definitions. 

Let $X=\{x_1,x_2,\ldots,x_n\}$ and $Y=\{y_1,y_2,\ldots,y_n\}$ be two subsets of $P$. We define the meet matrix on $X$ and $Y$ with respect to $f$ as $(X,Y)_f=(f(x_i\wedge y_j))$. In particular, when $S=X=Y=\{x_1,x_2,\ldots,x_n\}$, we have $(S,S)_f=(S)_f$. Analogously, we define the join matrix on $X$ and $Y$ with respect to $f$ as $[X,Y]_f=(f(x_i\vee y_j))$. In particular, $[S,S]_f=[S]_f$.

In \cite{ATH} the authors presented formulas for the determinant and the inverse of the matrix $(X,Y)_f$. Applying these formulas they derived similar formulas for the matrix $[X,Y]_{1/f}$ with respect to semimultiplicative functions $f$ with $f(x)\neq0$ for all $x\in P$. The cases when $f$ is not semimultiplicative or $f(x)=0$ for some $x\in P$, however, were excluded from the examination. In this paper we give formulas that can also be used in these circumstances. We go through the same examinations presented in \cite{ATH} but this time dually from the point of view of the matrix $[X,Y]_f$. That is we present expressions for the determinant and the inverse of $[X,Y]_f$ on arbitrary sets $X$ and $Y$. In the case when $X=Y=S$ we obtain a determinant formula for $[S]_f$ and a formula for the inverse of $[S]_f$ on arbitrary set $S$. We also derive formulas for the special cases when $S$ is join-closed and upper-closed up to $\vee S$. Similar kind of determinant formulas for $(S)_f$ and $[S]_f$ have already been presented in \cite{KH2}, although they were obtained and presented by a different approach. By setting $(P,\preceq)=(\Zset^+,|)$ we obtain corollaries for LCM matrices, and as another special case we also consider MAX and MIN matrices. In case when $(P,\preceq)=(\Zset,\leq)$, where $\leq$ is the natural ordering of the integers, the MAX and MIN matrices of the set $S$ are the matrices $[S]_f$ and $(S)_f$ respectively. MAX and MIN matrices have not been addressed before in this context.

\section{Preliminaries}

In the preceding section we defined the concept of GCD-closed set. Similarly, the set $S$ is said to be LCM-closed if $[x_i,x_j]\in S$ whenever $x_i,x_j\in S$. Since the lattice $(\Zset^+,|)$ does not have a greatest element, we need to define the dual concept for factor-closed set in a more special manner.

\begin{definition}\label{def1}
Let $\text{lcm}\,S=[x_1,x_2,\ldots,x_n]$, the least common multiple of all the elements in $S$. The set $S$ is said to be multiple-closed up to $\text{lcm}\,S$ if $x\in S$ whenever $y\in S$, $y\,|\,x$ and $x\,|\,\text{lcm}\,S$. In addition, let \[M_S=\{y\in\Zset^+\ \big|\ y\,|\,\text{lcm}\,S\ \text{and}\ x_i\,|\,y\ \text{for some}\ x_i\in S\}=\bigcup_{i=1}^n[x_i,\text{lcm}\,S],\]
where
\[[x_i,\text{lcm}\,S]=\{y\in\Zset^+\ \big|\ x_i\,|\,y\ \text{and}\ y\,|\,\text{lcm}\,S\}.\]
\end{definition}

Again, if $S$ is multiple-closed up to $\text{lcm}\,S$, then it is also LCM-closed, but an LCM-closed set is not necessarily multiple-closed up to $\text{lcm}\,S$. Obviously the set $M_S$ is multiple-closed up to $\text{lcm}\,S$, but the semilattice $(M_S,|)$ also has the advantage of having the greatest element over the lattice $(\Zset^+,|)$. Next we need corresponding definitions for a more general case.

Let $(P,\preceq)$ be a lattice. The set $S\subseteq P$ is said to be lower-closed (resp. upper-closed) if for every $x,y\in P$ with $x\in S$ and $y\preceq x$ (resp. $x\preceq y$), we have $y\in S$. The set $S$ is said to be meet-closed (resp. join-closed) if for every $x,y\in S$, we have $x\wedge y\in S$ (resp. $x\vee y\in S$).

If every principal order filter of the lattice $(P,\preceq)$ is finite, the methods presented in the following sections can be applied to the lattice $(P,\preceq)$ directly. If the lattice $(P,\preceq)$ does not satisfy this property (which is the case when, for example, $P=\Zset^+$ and $\preceq=|$), it is always possible to carry out the following procedures in an appropriate subsemilattice of $(P,\preceq)$. The most straightforward method is to generalize Definition \ref{def1}, which is done as follows.

\begin{definition}\label{def2}
Let $\vee S=x_1\vee x_2\vee\cdots\vee x_n$. The set $S$ is said to be upper-closed up to $\vee S$ if $x\in S$ whenever $y\in S$, $y\preceq x$ and $x\preceq \vee S$. In addition, let \[P_S=\{y\in L\ \big|\ y\preceq \vee S\ \text{and}\ x_i\preceq y\ \text{for some}\ x_i\in S\}=\bigcup_{i=1}^n[x_i,\vee S],\]
where
\[[x_i,\vee S]=\{y\in L\ \big|\ x_i\preceq y\ \text{and}\ y\preceq\vee S\}.\]
\end{definition}

Another possibility would be to restrict our consideration to $(\langle S\rangle,\preceq)$, the join-subsemilattice of $(P,\preceq)$ generated by the set $S$. Usually this would also reduce the number of computations needed. For example, the values of the Möbius function of $(\langle S\rangle,\preceq)$ are often much easier to calculate than the values of the Möbius function of $(P_S,\preceq)$ (see \cite[Section IV.1]{Aig}). And if we consider $S$ as a subset of the meet-subsemilattice generated by itself, the set $S$ is meet-closed iff it is lower-closed. Similarly, the terms join-closed and upper-closed coincide in the join-subsemilattice $(\langle S\rangle,\preceq)$. This is another benefit of restricting to $(\langle S\rangle,\preceq)$. This method is not, however, very convenient when considering the lattice $(\Zset^+,|)$. The Möbius function of $(\langle S\rangle,\preceq)$, where $S\subset \Zset^+$, has often very little in common with the number-theoretic Möbius function, which would likely cause confusion. For this reason we give our formulas in a form that fits both for the types of lattices defined in Definition \ref{def2} and for the lattice $(\langle S\rangle,\preceq)$.

Let $(P,\preceq)$ be a locally finite lattice, and let $f$ be a complex-valued function on $P$. Let $X=\{x_1,x_2,\ldots,x_n\}$ and $Y=\{y_1,y_2,\ldots,y_n\}$ be two subsets of $P$. Let the elements of $X$ and $Y$ be arranged so that $x_i\preceq x_j\Rightarrow i\leq j$. Let $D=\{d_1,d_2\ldots,d_m\}$ be any subset of $P$ containing the elements $x_i\vee y_j$, $i,j=1,2,\ldots n$. Let the elements of $D$ be arranged so that $d_i\preceq d_j\Rightarrow i\leq j$. Then we define the function $\Psi_{D,f}$ on $D$ inductively as
\begin{equation}\label{eq:psi1}
\Psi_{D,f}(d_k)=f(d_k)-\sum_{d_k\prec d_v}\Psi_{D,f}(d_v)
\end{equation}
or equivalently
\begin{equation}\label{eq:psi2}
f(d_k)=\sum_{d_k\preceq d_v}\Psi_{D,f}(d_v).
\end{equation}
Then
\begin{equation}\label{eq:psi3}
\Psi_{D,f}(d_k)=\sum_{d_k\preceq d_v}f(d_v)\mu_D(d_k,d_v),
\end{equation}
where $\mu_D$ is the Möbius function of the poset $(D,\preceq)$ (see \cite[3.7.2 Proposition]{St}).

If $D$ is join-closed, then 
\begin{equation}\label{eq:psijoin}
\Psi_{D,f}(d_{k})
=\sum_{d_{k}\preceq z\atop {d_t\npreceq z\atop k<t}}
                 \sum_{z\preceq w\preceq s} f(w)\mu_{P_D}(z, w),
\end{equation}
where $\mu_{P_D}$ is the Möbius function of $(P_D,\preceq)$, and if $D$ is upper-closed up to $\vee D$, then
{\begin{equation}\label{eq:psiupper}
\Psi_{D,f}(d_{k})=\sum_{d_k\preceq d_v}f(d_v)\mu_{P_D}(d_k,d_v),
\end{equation}}
where $\mu_{P_D}$ is the Möbius function of $(P_D,\preceq)$. The proofs of formulas \eqref{eq:psijoin} and \eqref{eq:psiupper} are dual to the proofs presented in \cite{H}.

\begin{remark}\label{re:sub}
If $D$ is join-closed, then $D=\langle D\rangle$ and $D$ is trivially upper-closed subset of $\langle D\rangle$ in $(\langle D\rangle,\preceq)$. Thus in this case we could also replace the $\mu_D$ in \eqref{eq:psi3} and the $\mu_{P_D}$ in \eqref{eq:psiupper} by $\mu_{\langle D\rangle}$.
\end{remark}

If $(P,\preceq)=(\Zset^+,|)$ and $D$ is multiple-closed up to $\text{lcm}\,D$, then $\mu_D(d_k,d_v)=\mu(d_v/d_k)$ (see \cite[Chapter 7]{M}), where $\mu$ is the number-theoretic Möbius function. In addition, for every $a\in\Zset^+$ and arithmetical function $f$ we may define another arithmetical function $f_a$, where
\[
f_a(n)=f(an)
\]
for every $n\in\Zset^+$. Now from \eqref{eq:psi3} we get
{\begin{align}\label{eq:psimultiple}
\Psi_{D,f}(d_k)&=\sum_{d_k\,|\,d_v}f(d_v)\mu\left(\frac{d_v}{d_k}\right)=\sum_{a\,|\,\frac{\text{lcm}\,D}{d_k}}f(d_ka)\mu(a)\notag\\ 
&=\sum_{a\,|\,\frac{\text{lcm}\,D}{d_k}}(f_{d_k}\mu)(a)=[\zeta\ast(f_{d_k}\mu)]\left(\frac{\text{lcm}\,D}{d_k}\right),
\end{align}}
where $*$ is the Dirichlet convolution of arithmetical functions.

Let  $E(X)=(e_{ij}(X))$ and $E(Y)=(e_{ij}(Y))$ denote the $n\times m$
matrices defined by
\begin{equation}\label{eq:eX}
e_{ij}(X)=\left\{
 \begin{array}{cc}
    1 & \textrm{if }x_{i}\preceq d_{j}\textrm{,} \\
    0 & \textrm{otherwise,}
 \end{array}
\right.
\end{equation}
and
\begin{equation}\label{eq:eY}
e_{ij}(Y)=\left\{
 \begin{array}{cc}
    1 & \textrm{if }y_{i}\preceq d_{j}\textrm{,} \\
    0 & \textrm{otherwise}
 \end{array}
\right.
\end{equation}
respectively.
Clearly $E(X)$ and $E(Y)$ also depend on $D$ but for the sake of brevity
$D$ is omitted from the notation. We also denote
\begin{equation}\label{eq:lambda}
\Lambda_{D,f}=\mathrm{diag}(\Psi_{D,f}(d_{1}),\Psi_{D,f}(d_{2}),
\ldots ,\Psi_{D,f}(d_{m})).
\end{equation}

\setcounter{equation}{0}
\section{A structure theorem}\label{sec:str}

In this section we give a factorization of
the matrix $[X,Y]_{f}=\left(f(x_{i}\vee y_{j})\right)$. A large number of similar factorizations is presented in the literature, for example in \cite{IHM} the matrix $[S]_f$ is factorized in case when $S$ is join-closed. The idea of this kind of factorization may be considered to originate from
P\'olya and Szeg\"o \cite{PS}.

\begin{theorem}\label{th:fac}
\begin{equation}\label{eq:fac}
[X,Y]_{f}=E(X)\Lambda_{D, f}E(Y)^{T}.
\end{equation}
\end{theorem}

\begin{proof}
By (\ref{eq:psi2}) the $ij$ entry of $[X, Y]_{f}$ is
\begin{equation}\label{eq:lemma}
f(x_{i}\vee y_{j})
=\sum\limits_{x_{i}\vee y_{j}\preceq d_{v}}\Psi_{D, f}(d_{v}).
\end{equation}
Now, applying (\ref{eq:eX}), (\ref{eq:eY}) and (\ref{eq:lambda})
 to (\ref{eq:lemma}) we obtain Theorem \ref{th:fac}.
\end{proof} 

\begin{remark}\upshape
The sets $X$ and $Y$ could be allowed to have distinct cardinalities
in Theorems \ref{th:fac} and \ref{jth:fac}.
However, in other results we must assume that these cardinalities coincide.
\end{remark}

\setcounter{equation}{0}
\section{Determinant formulas}\label{sec:det}

In this section we derive formulas for determinants of join matrices.
In Theorem \ref{th:CB} we present an expression for
$\det[X,Y]_{f}$ on arbitrary sets $X$ and $Y$.
Taking $X=Y=S=\{x_{1},x_{2},\ldots ,x_{n}\}$
in Theorem \ref{th:CB} we obtain
a formula for the determinant of usual join matrices $[S]_{f}$ on arbitrary
set $S$,
and further taking $(P,\preceq)=(\Zset^+, |)$ we obtain a formula
for the determinant of LCM matrices on arbitrary set $S$.
In Theorems \ref{th:detSjoin} and \ref{th:detSupper} respectively,
we calculate $\det[S]_{f}$ when $S$ is join-closed and upper-closed up to $\vee S$.
Formulas similar to Theorems \ref{th:detSjoin} and \ref{th:detSupper} but by different approach and notations are given in \cite{KH2}.

\begin{theorem}\label{th:CB}
{\rm (i)} If $n>m$, then $\det[X,Y]_{f}=0$.

\noindent
{\rm (ii)} If $n\leq m$, then
\begin{eqnarray}
\det[X,Y]_{f}
&=&\sum\limits_{1\leq k_{1}<k_{2}<\cdots <k_{n}\leq m}
   \det E(X)_{(k_{1},k_{2},\ldots ,k_{n})}
   \det E(Y)_{(k_{1},k_{2},\ldots ,k_{n})}\nonumber\\
&&\hspace*{3cm}\times \Psi_{D,f}(d_{k_{1}})\Psi_{D,f}(d_{k_{2}})\cdots
   \Psi_{D,f}(d_{k_{n}}).
\end{eqnarray}
\end{theorem}

\begin{proof}
By Theorem \ref{th:fac}
\begin{equation}
\det [X,Y]_{f}=
\det\left(E(X)\Lambda_{D,f}E(Y)^{T}\right).
\end{equation}
Thus by the Cauchy-Binet formula we obtain Theorem \ref{th:CB}.
\end{proof}

\begin{theorem}\label{th:detSjoin}
If $S$ is join-closed, then
\begin{align}\label{eq:detSjoin}
\det[S]_f&=\prod_{v=1}^n  \Psi_{S,f}(x_{v})
=\prod_{v=1}^n \sum_{x_v\preceq x_t}f(x_t)\mu_S(x_v,x_t)\notag\\ 
&=\prod_{v=1}^n
 \sum_{x_v\preceq z\atop {x_t\,\not\preceq\,z\atop v<t}}
                 \sum_{z\preceq w\preceq \vee S} f(w)\mu_{P_S}(z, w).
\end{align}
\end{theorem}

\begin{proof}
We take $X=Y=S$ in Theorem \ref{th:CB}.
Since $S$ is join-closed, we may further take $\langle D\rangle=D=S$.
Then $m=n$ and
$\det E(S)_{(k_{1},k_{2},\ldots ,k_{n})}
=\det E(S)_{(1, 2,\ldots , n)}=1$ and
so we obtain the first equality in $(\ref{eq:detSjoin})$.
The second equality follows from Remark \ref{re:sub} and the third from $(\ref{eq:psijoin})$.
\end{proof}

\begin{remark}\upshape
Theorem \ref{th:detSjoin} can also be proved
by taking $X=Y=S$ and $D=S$ in Theorem \ref{th:fac}.
\end{remark}

\begin{example}\label{ex1}
Let $(P,\preceq)=(\Zset,\leq)$, where $\leq$ is the natural ordering of the set of integers, and let $S=\{x_1,x_2,\ldots,x_n\}\subset\Zset$, where $x_1<x_2<\cdots<x_n$.  Let $t\in\Cset$ and $f:\Zset\to\Cset$ be such function that $f(k)=k+t$ for all $k\in\Zset$. Since the lattice $(\Zset,\leq)$ is a chain, the set $S$ is trivially both meet and join-closed. Now it follows from Theorem \ref{th:detSjoin} that the determinant of the MAX matrix $[S]_f$ is
\begin{align}
\det[S]_f&=\prod_{v=1}^n \sum_{x_v\preceq x_t}f(x_t)\mu_S(x_v,x_t)\notag\\ 
&=(f(x_1)-f(x_2))(f(x_2)-f(x_3))\cdots(f(x_{n-1})-f(x_n))f(x_n)\notag\\
&=(x_1-x_2)(x_2-x_3)\cdots(x_{n-1}-x_n)(x_n+t).
\end{align}
\end{example}

\begin{theorem}\label{th:detSupper}
If $S$ is upper-closed up to $\vee S$, then
\begin{equation}\label{eq:detSupper}
\det[S]_f=\prod_{v=1}^n \Psi_{S,f}(x_{v})
=\prod_{v=1}^n \sum_{x_v\preceq x_u} f(x_u)\mu(x_v, x_u).
\end{equation}
\end{theorem}

\begin{proof}
The first equality in $(\ref{eq:detSupper})$
follows from $(\ref{eq:detSjoin})$.
The second equality follows from  $(\ref{eq:psiupper})$.
\end{proof}

\begin{example}\label{ex2}
Let $(P,\preceq)$, $S$ and $f$ be as in Example \ref{ex1} and let $x_i=x_1+(i-1)$ for every $x_i\in S$. Now the set $S$ is clearly upper-closed up to $x_n=x_1+(n-1)$ and from Theorem \ref{th:detSupper} we get the determinant of the MAX matrix $[S]_f$ as
\begin{align}
\det[S]_f&=\prod_{v=1}^n \sum_{x_v\preceq x_t}f(x_t)\mu(x_v,x_t)\notag\\ 
&=(f(x_1)-f(x_2))(f(x_2)-f(x_3))\cdots(f(x_{n-1})-f(x_n))f(x_n)\notag\\
&=(-1)^{n-1}(x_n+t).
\end{align}

\end{example}

\begin{corollary}\label{co:HS}
%{\rm \cite{}}.
Let $(P,\preceq)=(\Zset^+,|)$, let $S$ be an LCM-closed set of distinct positive integers, and let
$f$ be an arithmetical function. Then the determinant of the LCM matrix $[S]_f$ is
\begin{equation}\label{eq:HS}
\det[S]_f=\prod\limits_{v=1}^n
	\sum\limits_{x_v\,\mid\,z\,\mid\,\mathrm{lcm}\,S\atop {x_t\,\nmid\,z\atop v<t}}
	[\zeta\ast(f_z\mu)]\left(\frac{\mathrm{lcm}\,S}{z}\right).
\end{equation}
\end{corollary}

\begin{corollary}\label{co:Sm}
%{\rm \cite{}}.
Let $(P,\preceq)=(\Zset^+,|)$, let $S$ be a set of distinct positive integers which is multiple-closed up to $\mathrm{lcm}\,S$, and let
$f$ be an arithmetical function.
Then
\begin{equation}\label{eq:Sm}
\det[S]_f=\prod\limits_{v=1}^n
  [\zeta\ast(f_{x_v}\mu)]\left(\frac{\mathrm{lcm}\,S}{x_v}\right).
\end{equation}
\end{corollary}

\setcounter{equation}{0}
\section{Inverse formulas}\label{sec:inv}

In this section we derive formulas for inverses of join matrices.
In Theorem \ref{th:invXY} we present an expression for
the inverse of $[X,Y]_{f}$ on arbitrary sets $X$ and $Y$,
and in Theorem \ref{th:invS} we present an expression for
the inverse of $[S]_{f}$ on arbitrary set $S$.
Taking $(P,\preceq)=(\Zset^+, |)$ we obtain a formula
for the inverse of LCM matrices on arbitrary set $S$.
Such formulas for the inverse of join or LCM matrices on an arbitrary set
have not previously been presented in the literature.
In Theorem \ref{th:invSjoin} we calculate the inverse of $[S]_{f}$ on join-closed set $S$ and in Theorem \ref{th:invSupper} we cover the case in which $S$ is upper-closed up to $\vee S$.

\begin{theorem}\label{th:invXY}
Let $X_i=X\setminus\{x_i\}$ and $Y_i=Y\setminus\{y_i\}$ for
$i=1, 2,\ldots, n$.
If $[X,Y]_{f}$ is invertible, then
the inverse of $[X,Y]_{f}$ is the $n\times n$ matrix $B=(b_{ij})$, where
\begin{eqnarray}
b_{ij}
&=&{(-1)^{i+j}\over \det[X,Y]_{f}}
\sum\limits_{1\leq k_{1}<k_{2}<\cdots <k_{n-1}\leq m}
   \det E(X_j)_{(k_{1},k_{2},\ldots ,k_{n-1})}
   \det E(Y_i)_{(k_{1},k_{2},\ldots ,k_{n-1})}\nonumber\\
&&\hspace*{52mm}\times \Psi_{D,f}(d_{k_{1}})\Psi_{D,f}(d_{k_{2}})\cdots
   \Psi_{D,f}(d_{k_{n-1}}).
\end{eqnarray}
\end{theorem}

\begin{proof}
It is well known that
\begin{equation}
b_{ij}={\alpha_{ji}\over \det[X,Y]_{f}},
\end{equation}
where $\alpha_{ji}$ is the cofactor of the $ji$-entry of $[X,Y]_{f}$.
It is easy to see that
$\alpha_{ji}=(-1)^{i+j}\det[X_{j},Y_{i}]_{f}$.
By Theorem \ref{th:CB} we see that
\begin{eqnarray}
\det[X_{j},Y_{i}]_{f}
&=&\sum\limits_{1\leq k_{1}<k_{2}<\cdots <k_{n-1}\leq m}
   \det E(X_j)_{(k_{1},k_{2},\ldots ,k_{n-1})}
   \det E(Y_i)_{(k_{1},k_{2},\ldots ,k_{n-1})}\nonumber\\
&&\hspace*{32mm}\times \Psi_{D,f}(d_{k_{1}})\Psi_{D,f}(d_{k_{2}})\cdots
   \Psi_{D,f}(d_{k_{n-1}}).
\end{eqnarray}
Combining the above equations we obtain Theorem \ref{th:invXY}.
\end{proof}

\begin{theorem}\label{th:invS}
Let $S_i=S\setminus\{x_i\}$ for $i=1, 2,\ldots, n$.
If $[S]_{f}$ is invertible, then
the inverse of $[S]_{f}$ is the $n\times n$ matrix $B=(b_{ij})$, where
\begin{eqnarray}
b_{ij}
&=&{(-1)^{i+j}\over \det[S]_{f}}
\sum\limits_{1\leq k_{1}<k_{2}<\cdots <k_{n-1}\leq m}
   \det E(S_j)_{(k_{1},k_{2},\ldots ,k_{n-1})}
   \det E(S_i)_{(k_{1},k_{2},\ldots ,k_{n-1})}\nonumber\\
&&\hspace*{47mm}\times \Psi_{D,f}(d_{k_{1}})\Psi_{D,f}(d_{k_{2}})\cdots
   \Psi_{D,f}(d_{k_{n-1}}).
\end{eqnarray}
\end{theorem}

\begin{proof}
Taking $X=Y=S$ in Theorem \ref{th:invXY} we obtain
Theorem \ref{th:invS}.
\end{proof}

\begin{theorem}\label{th:invSjoin}
Suppose that $S$ is join-closed.
If $[S]_f$ is
invertible, then the inverse of $[S]_{f}$ is the $n\times n$ matrix
$B=(b_{ij})$, where
\begin{equation}
b_{ij}=\sum_{k=1}^{n}{(-1)^{i+j}\over\Psi_{S,f}(x_k)}
\det E(S_{i}^{k})\det E(S_{j}^{k}),
\end{equation}
where
$E(S_{i}^{k})$ is the $(n-1)\times (n-1)$ submatrix of $E(S)$ obtained by
deleting the $i$th row and the $k$th column of $E(S)$, or
\begin{equation}\label{eq:invSjoin}
b_{ij}
=
\sum_{x_k\preceq x_i\wedge x_j}
{\mu_S(x_k, x_i)\mu_S(x_k, x_j)\over \Psi_{S, f}(x_k)},
\end{equation}
where $\mu_S$ is the M\"{o}bius function of the poset $(S,\preceq)$.
\end{theorem}

\begin{proof}
Since $S$ is join-closed, we may take $D=S$.
Then $E(S)$ is a square matrix with $\det E(S)=1$.
Further, $E(S)$ is the
matrix associated with the zeta function of
the finite poset $(S,\preceq)$.
Thus the inverse of $E(S)$ is the
matrix associated with the M\"{o}bius function of
$(S,\preceq)$, that is,
if $U=(u_{ij})$ is the inverse of $E(S)$, then
$u_{ij}=\mu_S(x_i, x_j)$, see \cite[p. 139]{Aig}.
On the other hand, $u_{ij}=\beta_{ij}/\det E(S)=\beta_{ij}$,
where $\beta_{ij}$ is the cofactor of the $ij$-entry of $E(S)$.
Here
$\beta_{ij}=(-1)^{i+j}\det E(S_i^j)$.
Thus
\begin{equation}\label{eq:aid}
(-1)^{i+j}\det E(S_i^j)=\mu_S(x_i, x_j).
\end{equation}
Now we apply Theorem \ref{th:invS} with $D=S$.
Then $m=n$, and using formulas $(\ref{eq:detSjoin})$ and
$(\ref{eq:aid})$ we obtain Theorem \ref{th:invSjoin}.
\end{proof}

\begin{remark}\upshape
Equation (\ref{eq:invSjoin}) can also be
proved  by taking $X=Y=S$ and $D=S$ in Theorem \ref{th:fac}
and then applying the formula
$[S]_{f}^{-1}=(E(S)^{T})^{-1}\Lambda_{S,f}^{-1}E(S)^{-1}$.
\end{remark}

\begin{example}\label{ex3}
Let $(P,\preceq)$, $f$ and $S$ be as in Example \ref{ex1}. If $t\neq -x_n$, then the matrix $[S]_f$ is invertible and the inverse of $[S]_f$ is the $n\times n$ tridiagonal matrix
$B=(b_{ij})$, where
\[
b_{ij} = \left\{ \begin{array}{ll}
0 & \textrm{if $|i-j|>1$},\\
\frac{1}{x_1-x_2} & \textrm{if $i=j=1$},\\
\frac{1}{x_{i-1}-x_{i}}+\frac{1}{x_i-x_{i+1}} & \textrm{if $1<i=j<n$},\\
\frac{1}{x_{n-1}-x_{n}}+\frac{1}{x_n+t} & \textrm{if $i=j=n$},\\
\frac{1}{|x_i-x_j|} & \textrm{if $|i-j|=1$}.
\end{array} \right.
\]
\end{example}

\begin{theorem}\label{th:invSupper}
Suppose that $S$ is upper-closed up to $\vee S$.
If $[S]_f$ is
invertible, then the inverse of $[S]_{f}$ is the $n\times n$ matrix
$B=(b_{ij})$ with
\begin{equation}\label{eq:invSupper}
b_{ij}
=\sum_{x_k\preceq x_i\wedge x_j}
{\mu(x_k, x_i)\mu(x_k, x_j)\over \Psi_{S, f}(x_k)},
\end{equation}
where $\mu$ is the M\"{o}bius function of $(P,\preceq)$.
\end{theorem}

\begin{proof}
Since $S$ is upper-closed up to $\vee S$, we have $\mu_S=\mu$ on $S$,
(apply \cite[Proposition 4.6]{Aig}).  Thus Theorem \ref{th:invSupper} follows from
Theorem \ref{th:invSjoin}.
\end{proof}

\begin{example}\label{ex4}
Let $(P,\preceq)$, $f$ and $S$ be as in Example \ref{ex2}. If $t\neq -x_n$, then the matrix $[S]_f$ is invertible and the inverse of $[S]_f$ is the $n\times n$ tridiagonal matrix
$B=(b_{ij})$, where
\[
b_{ij} = \left\{ \begin{array}{ll}
0 & \textrm{if $|i-j|>1$},\\
-1 & \textrm{if $i=j=1$},\\
-2 & \textrm{if $1<i=j<n$},\\
-1+\frac{1}{x_n+t} & \textrm{if $i=j=n$},\\
1 & \textrm{if $|i-j|=1$}.
\end{array} \right.
\]
\end{example}

\begin{corollary}\label{co:BL2}
Let $S$ be a set of distinct
positive integers which is multiple-closed up to $\mathrm{lcm}\,S$, and let
$f$ be an arithmetical function.
If the LCM matrix $[S]_f$ is invertible,
then its inverse is the $n\times n$ matrix
$B=(b_{ij})$, where
\begin{equation}\label{eq:BL2}
b_{ij}
=\sum\limits_{ x_k \mid (x_i, x_j) }
{\mu(x_i/x_k)\mu(x_j/x_k)\over [\zeta\ast(f_{x_k}\mu)]\left(\frac{\mathrm{lcm}\,S}{x_k}\right)}.
\end{equation}
Here $\mu$ is the number-theoretic M\"{o}bius function.
\end{corollary}

\setcounter{equation}{0}
\section{Formulas for meet matrices}\label{sec:det}

Let $f$ be a complex-valued function on $P$.
We say that $f$ is a semimultiplicative function if
\begin{equation}\label{eq:semi}
f(x)f(y)=f(x\land y)f(x\lor y)
\end{equation}
for all $x,y\in P$ (see \cite{KH2}).

The notion of a semimultiplicative function arises from
the theory of arithmetical functions.
Namely, an arithmetical function $f$ is said to be semimultiplicative
if $f(r)f(s)=f((r, s))f([r, s])$ for all $r, s\in\Zset^+$.
For  semimultiplicative arithmetical functions
reference is made to the book by Sivaramakrishnan \cite{Si}, see also
\cite{HS}. Note that a semimultiplicative arithmetical function $f$
with $f(1)\neq 0$ is referred to as
a quasimultiplicative arithmetical function.
Quasimultiplicative arithmetical functions with $f(1)=1$ are the usual
multiplicative arithmetical functions.

In this section we show that meet matrices $(X,Y)_{f}$ with
respect to semimultiplicative functions $f$ possess properties similar
to those given for join matrices $[X,Y]_{f}$ with
respect to arbitrary functions $f$
in Sections 3, 4 and 5. Since there already are several formulas for the determinant and the inverse of the matrix $(X,Y)_{f}$ (see \cite{ATH} and \cite{KH4}), the motivation in deriving new formulas probably needs clarification. The formulas of this section are especially useful when considering the matrix $(S)_f$, where the set $S$ is either join-closed or upper closed up to $\vee S$. That is, because in this case the formulas of this section result in shorter and simplier calculations.
{\sl Throughout this section $f$ is a semimultiplicative function on $P$ such
that $f(x)\neq 0$ for all $x\in P$.}

\begin{theorem}\label{jth:fac}
\begin{equation}\label{jeq:fac}
(X,Y)_{f}=\Delta_{X,f}[X, Y]_{1/f}\Delta_{Y,f}
\end{equation}
or
\begin{equation}\label{jeq:fac2}
(X,Y)_{f}=\Delta_{X,f}E(X)\Lambda_{D,1/f}E(Y)^{T}\Delta_{Y,f},
\end{equation}
where
\begin{equation}\label{eq:deltax}
\Delta_{X,f}=\mathrm{diag}(f(x_{1}), f(x_{2}),\ldots , f(x_{n}))
\end{equation}
and
\begin{equation}\label{eq:deltay}
\Delta_{Y,f}=\mathrm{diag}(f(y_{1}), f(y_{2}),\ldots , f(y_{n})).
\end{equation}
\end{theorem}

\begin{proof}
By (\ref{eq:semi}) the $ij$-entry of $(X,Y)_{f}$ is
\begin{equation}
f(x_{i}\wedge y_{j})
=f(x_{i}) {1\over f(x_{i}\vee y_{j})} f(y_{j}).
\end{equation}
We thus obtain (\ref{jeq:fac}), and applying Theorem \ref{th:fac} we
obtain (\ref{jeq:fac2}).
\end{proof}

From (\ref{jeq:fac}) we obtain
\begin{equation}\label{jeq:dett}
\det(X,Y)_{f}=\left(\prod_{v=1}^n f(x_v)f(y_v)\right)\det[X,Y]_{1/f}
\end{equation}
and
\begin{equation}\label{jeq:invv}
(X,Y)_{f}^{-1}=\Delta_{Y,f}^{-1}[X, Y]_{1/f}^{-1}\Delta_{X,f}^{-1}.
\end{equation}
Now, using (\ref{jeq:dett}), (\ref{jeq:invv}) and the formulas of Sections 4
and 5  we obtain formulas for meet matrices.

We first present formulas for the determinant of meet matrices.
In Theorem \ref{jth:CB}
we give a formula for $\det(X,Y)_{f}$ on arbitrary sets $X$ and $Y$.
This is an alternative expression that given in \cite{ATH}. 
In Theorems \ref{jth:detSjoin} and \ref{jth:detSupper} respectively,
we calculate $\det(S)_{f}$ when $S$ is join-closed and upper-closed up to $\vee S$.

\begin{theorem}\label{jth:CB}
{\rm (i)} If $n>m$, then $\det(X,Y)_{f}=0$.

\noindent
{\rm (ii)} If $n\leq m$, then
\begin{eqnarray}
\det(X,Y)_{f}
&=&\left(\prod_{v=1}^n f(x_v)f(y_v)\right)
   \Biggl(\sum\limits_{1\leq k_{1}<k_{2}<\cdots <k_{n}\leq m}
   \det E(X)_{(k_{1},k_{2},\ldots ,k_{n})}
   \det E(Y)_{(k_{1},k_{2},\ldots ,k_{n})}\nonumber\\
&&\times \Psi_{D,1/f}(d_{k_{1}})\Psi_{D,1/f}(d_{k_{2}})\cdots
   \Psi_{D,1/f}(d_{k_{n}})\Biggr).
\end{eqnarray}
\end{theorem}

\begin{theorem}\label{jth:detSjoin}
If $S$ is join-closed, then
\begin{align}\label{jeq:detSjoin}
\det(S)_f&=\prod_{v=1}^n f(x_v)^2 \Psi_{S,1/f}(x_{v})=\prod_{v=1}^n f(x_v)^2\sum_{x_v\preceq x_u}{\frac{\mu_S(x_v, x_u)}{f(x_u)}}\notag\\
&=\prod_{v=1}^n f(x_v)^2
 \sum_{x_v\preceq z\atop {x_t\,\not\preceq\,z\atop v<t}}
                 \sum_{z\preceq w\preceq\vee S} {\frac{\mu(z, w)}{f(w)}}.
\end{align}
\end{theorem}

\begin{example}\label{ex5}
Let $(P,\preceq)=(\Zset,\leq)$, $t\in\Cset$ a complex number such that $t\neq-x_i$ for all $x_i\in S$ and $f(x_i)=x_i+t$ for all $x_i\in S$. Since $(\Zset,\leq)$ is a chain, the function $f$ is trivially semimultiplicative. Now from Theorem \ref{jth:detSjoin} we get
\begin{align*}
\det(S)_f&=\prod_{v=1}^n f(x_v)^2
 \sum_{x_v\preceq x_u}{\frac{\mu_S(x_v, x_u)}{f(x_u)}}
 =\prod_{v=1}^n f(x_v)^2\left(\frac{1}{f(x_v)}-\frac{1}{f(x_{v+1})}\right)\\
 &=\prod_{v=1}^n f(x_v)^2\frac{f(x_{v+1})-f(x_v)}{f(x_v)f(x_{v+1})}\\
 &=f(x_1)(f(x_2)-f(x_1))(f(x_3)-f(x_2))\cdots(f(x_n)-f(x_{n-1}))\\
 &=(x_1+t)(x_2-x_1)(x_3-x_2)\cdots(x_n-x_{n-1}).
 \end{align*}
\end{example}

\begin{theorem}\label{jth:detSupper}
If $S$ is upper-closed up to $\vee S$, then
\begin{equation}\label{jeq:detSupper}
\det(S)_f=\prod_{v=1}^n f(x_v)^2 \Psi_{S,1/f}(x_{v})
=\prod_{v=1}^n f(x_v)^2
 \sum_{x_v\preceq x_u} {\frac{\mu(x_v, x_u)}{f(x_u)}}.
\end{equation}
\end{theorem}

\begin{example}\label{ex6}
Let $(P,\preceq)=(\Zset,\leq)$, $S=\{x_1,x_1+1,x_1+2,\ldots,x_1+n-1\}$, $t\in\Cset$ a complex number such that $t\neq-x_i$ for all $x_i\in S$ and $f(x_i)=x_i+t$ for all $x_i\in S$. Now it follows from Theorem \ref{jth:detSupper} that
\begin{align*}
\det(S)_f&=\prod_{v=1}^n f(x_v)^2
 \sum_{x_v\preceq x_u}{\frac{\mu(x_v, x_u)}{f(x_u)}}
 =\prod_{v=1}^n f(x_v)^2\left(\frac{1}{f(x_v)}-\frac{1}{f(x_{v+1})}\right)\\
 &=\prod_{v=1}^n f(x_v)^2\frac{f(x_{v+1})-f(x_v)}{f(x_v)f(x_{v+1})}\\
 &=f(x_1)(f(x_2)-f(x_1))(f(x_3)-f(x_2))\cdots(f(x_n)-f(x_{n-1}))\\
 &=(x_1+t)(-1)^{n-1}.
 \end{align*}
\end{example}

\begin{corollary}\label{jco:HS}
Let $S$ be an LCM-closed set of distinct positive integers, and let
$f$ be a quasimultiplicative arithmetical function
such that $f(r)\neq0$ for all $r\in\Zset^+$. Then
\begin{equation}\label{jeq:HS}
\det(S)_f=\prod\limits_{v=1}^n
  f(x_v)^2\sum\limits_{x_v\,\mid\,z\,\mid\,\mathrm{lcm}\,S \atop {x_t\,\nmid\,z\atop v<t}}
\left[\zeta\ast\left(\frac{\mu}{f_z}\right)\right]\left(\frac{\mathrm{lcm}\,S}{z}\right).
\end{equation}
\end{corollary}

\begin{corollary}\label{jco:BL}
Let $S$ be a set of distinct positive integers which is multiple-closed up to $\mathrm{lcm}\,S$, and let
$f$ be a quasimultiplicative arithmetical function
such that $f(r)\neq0$ for all $r\in\Zset^+$.
Then
\begin{equation}\label{jeq:BL}
\det(S)_f=\prod\limits_{v=1}^n
  f(x_v)^2
\left[\zeta\ast\left(\frac{\mu}{f_{x_v}}\right)\right]\left(\frac{\mathrm{lcm}\,S}{x_v}\right).
\end{equation}
\end{corollary}

We next derive formulas for inverses of meet matrices.
In Theorem \ref{jth:invXY} we give an expression for the inverse of $(X,Y)_{f}$
on arbitrary sets $X$ and $Y$, and in Theorem \ref{jth:invS}
we give an expression for the inverse of
$(S)_{f}$ on arbitrary set $S$.
Taking $(P,\preceq)=(\Zset^+, |)$ we could obtain a formula
for the inverse of GCD matrices on arbitrary set $S$.
In Theorems \ref{jth:invSjoin} and \ref{jth:invSupper}, respectively,
we calculate the inverse of $(S)_{f}$ in cases when $S$ is join-closed and upper-closed up to $\vee S$. Formulas similar to Theorems \ref{jth:invSjoin} and \ref{jth:invSupper}, although with stronger assumptions, have been presented earlier in \cite{KH2}.

\begin{theorem}\label{jth:invXY}
Let $X_i=X\setminus\{x_i\}$ and $Y_i=Y\setminus\{y_i\}$ for
$i=1, 2,\ldots, n$.
If $[X,Y]_{f}$ is invertible, then
the inverse of $(X,Y)_{f}$ is the $n\times n$ matrix $B=(b_{ij})$ with
\begin{eqnarray}
b_{ij}
&=&{(-1)^{i+j}\over f(x_j)f(y_i) \det(X,Y)_{f}}
   \left(\prod_{v=1}^n f(x_v)f(y_v)\right)\nonumber\\
&&\times\Biggl(\sum\limits_{1\leq k_{1}<k_{2}<\cdots <k_{n-1}\leq m}
   \det E(X_j)_{(k_{1},k_{2},\ldots ,k_{n-1})}
   \det E(Y_i)_{(k_{1},k_{2},\ldots ,k_{n-1})}\nonumber\\
&&\times \Psi_{D,1/f}(d_{k_{1}})\Psi_{D,1/f}(d_{k_{2}})\cdots
   \Psi_{D,1/f}(d_{k_{n-1}})\Biggr).
\end{eqnarray}
\end{theorem}

\begin{theorem}\label{jth:invS}
Let $S_i=S\setminus\{x_i\}$ for $i=1, 2,\ldots, n$.
If $(S)_{f}$ is invertible, then
the inverse of $(S)_{f}$ is the $n\times n$ matrix $B=(b_{ij})$ with
\begin{eqnarray}
b_{ij}
&=&{(-1)^{i+j}\over f(x_i)f(x_j) \det(S)_{f}}
   \left(\prod_{v=1}^n f(x_v)^2\right)\nonumber\\
&&\times\Biggl(\sum\limits_{1\leq k_{1}<k_{2}<\cdots <k_{n-1}\leq m}
   \det E(S_i)_{(k_{1},k_{2},\ldots ,k_{n-1})}
   \det E(S_j)_{(k_{1},k_{2},\ldots ,k_{n-1})}\nonumber\\
&&\times \Psi_{D,1/f}(d_{k_{1}})\Psi_{D,1/f}(d_{k_{2}})\cdots
   \Psi_{D,1/f}(d_{k_{n-1}})\Biggr).
\end{eqnarray}
\end{theorem}

\begin{theorem}\label{jth:invSjoin}
Suppose that $S$ is join-closed.
If $(S)_f$ is
invertible, then the inverse of $(S)_{f}$ is the $n\times n$ matrix
$B=(b_{ij})$ with
\begin{equation}\label{jeq:invSjoin}
b_{ij}
={1\over f(x_i)f(x_j)}
\sum_{x_k\preceq x_i\wedge x_j}
{\mu_S(x_k, x_i)\mu_S(x_k, x_j)\over \Psi_{S,1/f}(x_k)}.
\end{equation}
Here $\mu_S$ is the M\"{o}bius function of the poset $(S,\preceq)$.
\end{theorem}

\begin{example}\label{ex7}
By Theorem \ref{jth:invSjoin}, the inverse of the MIN matrix $(S)_f$ in Example \ref{ex5} is the $n\times n$ tridiagonal matrix $B=(b_{ij})$ with
\[
b_{ij} = \left\{ \begin{array}{ll}
0 & \textrm{if $|i-j|>1$},\\
\frac{1}{x_{2}-x_{1}}\frac{x_2+t}{x_1+t} & \textrm{if $i=j=1$},\\
\frac{1}{x_i+t}\left(\frac{x_{i-1}+t}{x_i-x_{i-1}}+\frac{x_{i+1}+t}{x_{i+1}-x_i}\right) & \textrm{if $1<i=j<n$},\\
\frac{1}{x_n+t}\left(\frac{x_{n-1}+t}{x_n-x_{n-1}}+1\right) & \textrm{if $i=j=n$},\\
\frac{-1}{|x_i-x_j|} & \textrm{if $|i-j|=1$}.
\end{array} \right.
\]
\end{example}

\begin{theorem}\label{jth:invSupper}
Suppose that $S$ is
upper-closed up to $\vee S$.
If $(S)_f$ is
invertible, then the inverse of $(S)_{f}$ is the $n\times n$ matrix
$B=(b_{ij})$, where
\begin{equation}\label{jeq:invSlower}
b_{ij}
={1\over f(x_i)f(x_j)}
\sum_{x_k\preceq x_i\wedge x_j}
{\mu(x_k, x_i)\mu(x_k, x_j)\over \Psi_{S,1/f}(x_k)}.
\end{equation}
Here $\mu$ is the M\"{o}bius function of $(P,\preceq)$.
\end{theorem}

\begin{example}\label{ex8}
By Theorem \ref{jth:invSupper}, the inverse of the MIN matrix $(S)_f$ in Example \ref{ex6} is the $n\times n$ tridiagonal matrix $B=(b_{ij})$, where
\[
b_{ij} = \left\{ \begin{array}{ll}
0 & \textrm{if $|i-j|>1$},\\
\frac{x_1+1+t}{x_1+t} & \textrm{if $i=j=1$},\\
2 & \textrm{if $1<i=j<n$},\\
1 & \textrm{if $i=j=n$},\\
-1 & \textrm{if $|i-j|=1$}.
\end{array} \right.
\]
\end{example}

\begin{corollary}\label{jco:BL2}
Let $S$ be a set of distinct
positive integers which is multiple-closed up to $\mathrm{lcm}\,S$, and let
$f$ be a quasimultiplicative arithmetical function
such that $f(r)\neq0$ for all $r\in\Zset^+$.
If the GCD matrix $(S)_f$ is invertible,
then its inverse is the $n\times n$ matrix
$B=(b_{ij})$, where
\begin{equation}\label{jeq:BL2}
b_{ij}
={1\over f(x_i)f(x_j)}\sum\limits_{ x_k\,\mid\,(x_i, x_j) }
{\mu(x_i/x_k)\mu(x_j/x_k)\over \left[\zeta\ast\left(\frac{\mu}{f_{x_k}}\right)\right]\left(\frac{\mathrm{lcm}\,S}{x_k}\right)}.
\end{equation}
Here $\mu$ is the number-theoretic M\"{o}bius function.
\end{corollary}

\end{document}